\documentclass[12pt,reqno]{amsart}
\usepackage{amsmath}
\usepackage{amssymb,latexsym}
\usepackage[matrix,arrow]{xy}
\usepackage[usenames]{color}
\usepackage{slashed}
\usepackage{parskip}
\newtheorem{thm}{Theorem}

\newtheorem{lem}[thm]{Lemma}

\theoremstyle{definition}

\numberwithin{equation}{section}
 \theoremstyle{remark}

\def\R{\mathbb{R}}
\def\S{\mathbb{S}^3(1)}
\def\dt{\partial_t}
\def\dx{\partial_x}
\def\dy{\partial_y}
\def\E{\mathbb{E}}
\def\H{\mathbb{H}}
\def\ta{\tilde A_N}
\def\dpp{\tilde{\nabla}^\perp}

\newcommand{\Sf}{\mathbb{S}}

\newcommand{\Hy}{\mathbb{H}}
\begin{document}
\title{\text  Constant Angle Surfaces in $\mathbb{S}^3(1) \times \R$}
\author []{ Daguang Chen, Gangyi Chen, Hang Chen and Franki Dillen}
\email{dgchen@math.tsinghua.edu.cn}
\email{chen-gy07@mails.tsinghua.edu.cn}
\email{chenhang08@mails.tsinghua.edu.cn}
\address{Department of
Mathematical Sciences, Tsinghua University, Beijing, 100084, P.R.
China}

\email{franki.dillen@wis.kuleuven.be}
\address{Departement Wiskunde, Katholieke Universiteit Leuven,
Celestijnenlaan 200 B, B-3001 Leuven, Belgium}

\thanks{This research was supported by Tsinghua University and K.U.Leuven Bilateral scientific cooperation Fund, project BIL09/10}.



\begin{abstract}In this article we study surfaces in $\S\times\R$ for which the $\R$-direction makes a constant angle with the normal plane.
We give a complete classification for such surfaces with parallel mean curvature vector.
\end{abstract}
\maketitle
\renewcommand{\sectionmark}[1]{}

\section{introduction}
In recent years, there has been done some research about surfaces in
a 3-dimensional Riemannian product of a surface $\mathbb{M}^2(c)\times
\R$ (\cite{AR04, FM07, MR04, Ros02}, etc.), where $\mathbb{M}^2(c)$
is the simply-connected 2-dimensional space form of constant curvature $c$, in particular
$\mathbb{M}^2(c)=\R^2, ~\mathbb{H}^2, ~\mathbb{S}^2$ for $c=0, ~-1, ~1$
respectively.

Recently, constant angle surfaces were studied in product spaces
$\mathbb{M}^2(c)\times \R$ (see \cite{DFV09,DFVV07, DM07,DM09,
MN09,Nis09}), where the  angle was considered between the unit
normal of the surface $M$ and the tangent direction to $\mathbb{R}$.
For example, F. Dillen et al. gave the complete classification for
constant angle surfaces in $\mathbb{S}^{2}\times \R$ in
\cite{DFVV07}. The problem of constant angle surfaces was also
investigated in the 3-dimensional Heisenberg group (see
\cite{FMV09}) and in Minkowski space (see \cite{LM09}). In
\cite{To10}, R. Tojeiro gave a complete description of all
hypersurfaces in the product spaces $\Sf^n\times \R$ and
$\Hy^n\times \R$ that have flat normal bundle when regarded as
submanifolds with codimension two of the underlying flat spaces
$\R^{n+2}\supset \Sf^n\times \R$ and $\mathbb{L}^{n+2}\supset
\Hy^n\times \R$. In \cite{DR10}, helix submanifolds in Euclidean space were studied by solving the Eikonal equation.
The applications of constant angle surfaces in the
theory of liquid crystals and of layered fluids were considered by P. Cermelli and A. J. Di Scala in \cite{CdS07}.

In this article we study surfaces in $\S\times\R$ for which the
$\R$-direction makes a constant angle with the normal plane. In
Section 2, we first review some basic equations for constant angle
surfaces in $\S \times \R$. In Section 3, we will prove that the
constant angle surfaces in $\S\times \R$ with parallel
mean curvature vector are minimal (see Theorem \ref{thm1}). In
Section 4, we will give a complete classification for minimal and
constant angle surfaces in $\S \times \R$ (see Theorem
\ref{thm2}).

\section[Preliminaries]{Preliminaries}

Let $\widetilde M=\S\times\R$ be the Riemannian product of $\S$ and
$\R$ with the standard metric $\langle,\rangle$ and the Levi-Civita
connection $\widetilde\nabla$. We denote by $t$ the (global)
coordinate on $\R$ and hence $\dt=\frac\partial{\partial t}$ is the
unit vector field in the tangent bundle $T(\S\times\R)$ that is
tangent to the $\R$-direction.

For $p\in\S\times\R$, the Riemann-Christoffel curvature tensor
$\tilde{R}$ of $\S\times\R$ is given by
\[\langle\tilde{R}(X,Y)Z,W\rangle=\langle X_{\S},W_{\S}\rangle\langle Y_{\S},Z_{\S}\rangle-\langle X_{\S},Z_{\S}\rangle\langle Y_{\S},W_{\S}\rangle,\]
where $X,Y,Z,W\in T_p(\S\times\R)$ and $X_{\S}=X-\langle
X,\dt\rangle\dt$ is the projection of $X$ to the tangent space of
$\S$.

Now consider a surface $M$ in $\S\times\R$. We can decompose $\dt$ as
\begin{eqnarray}
\label{eq01}\dt=\sin\theta T+\cos\theta \xi,
\end{eqnarray}
where $\theta$ is
the angle between $\xi$ and $\dt$, $\xi$ is a unit normal vector to $M$ and $T$ is a unit tangent vector to $M$.

For a constant angle surface $M$ in $\S\times\R$, we mean a surface
for which the angle function $\theta$ is constant on $M$. There are
two trivial cases, $\theta=0$ and $\theta=\frac{\pi}{2}$. The
condition $\theta=0$ means that $\dt$ is always normal, so we get a surface
$\Sigma^2\times\{t_0\}$, where $\Sigma^2$ is a surface in $\S$. In
the second case, $\dt$ is always tangent. This corresponds to the
Riemannian product of a curve in $\S$ and $\R$.

From now on, in the rest of this paper, we only consider the
constant angle surface $M$ with constant angle
$\theta\in(0,\frac{\pi}{2})$. We extend $\{T, \xi\}$ to an
orthonormal frame $\{T, Q, \xi, \eta\}$ on $\S\times\R$, where $T,
Q$ are tangent to $M$ and $\xi, \eta$ are normal to $M$. Since $\dt$
is a parallel vector field in $\S\times\R$, we can obtain from
(\ref{eq01}) that, for any $X\in TM$,
\begin{equation}\label{parallel}
0=\widetilde\nabla_{X}\dt=\sin\theta \nabla_{X}T+\sin\theta
h(X,T)-\cos\theta A_{\xi}X+\cos\theta \nabla_{X}^{\perp}\xi,
\end{equation}
where we use the formulas of Gauss and Weingarten,  $h$ is the second fundamental form of $M$,
 $A_{\xi}$ is the shape operator associated to $\xi$, and $\nabla^{\perp}$ is the normal connection.

Comparing the tangent part and the normal part in (\ref{parallel}),
we have
\begin{equation}\label{eq02}
\left \{\aligned
 \nabla_{X}T&=\cot\theta A_{\xi}X,\\
h(X,T)&=-\cot\theta \nabla_{X}^{\perp}\xi.
\endaligned\right.
\end{equation}
From (\ref{eq02}), we have
\begin{equation*}
\langle A_{\xi}X, T\rangle=\langle A_{\xi}T, X\rangle=0,\quad
\forall X\in TM,
\end{equation*}
that is,
\begin{eqnarray*}
    A_{\xi}T =0.
 \end{eqnarray*}

Therefore, we can suppose the shape operators  with respect to
$\xi$ and $\eta$ are, respectively,
\begin{align}\label{CAS}
 A_{\xi}=\left(
 \begin{array}{cc}
 0 & 0\\
 0 & \lambda
 \end{array}\right),\quad
 A_{\eta}=\left(
 \begin{array}{cr}
 \beta_{1} & \beta_{2}\\
 \beta_{2} & \beta_{3}
 \end{array}\right),
\end{align}
where $\lambda, \beta_{j}~( j=1,2,3)$ are smooth functions defined
on the surface $M$.

From (\ref{eq02}),  we obtain that
\begin{equation}\label{p1}
\left \{\aligned
\nabla_{T}T&=\nabla_{T}Q=0,\\
\nabla_{Q}T&=\lambda\cot\theta Q,\\
\nabla_{Q}Q&=-\lambda\cot\theta T,
\endaligned\right.
\end{equation}
\begin{equation}\label{h1}
\left \{\aligned
h(T,T)&=\beta_{1}\eta,\\
h(T,Q)&=\beta_{2}\eta,\\
h(Q,Q)&=\lambda\xi+\beta_{3}\eta,
\endaligned\right.
\end{equation}
\begin{equation}\label{n1}
\left \{\aligned
 \nabla_{T}^{\perp}\xi &=-\tan\theta\,\beta_{1}\eta,\\
 \nabla_{T}^{\perp}\eta &=\tan\theta\,\beta_{1}\xi,\\
 \nabla_{Q}^{\perp}\xi &=-\tan\theta\,\beta_{2}\eta,\\
 \nabla_{Q}^{\perp}\eta &=\tan\theta\,\beta_{2}\xi.
\endaligned\right.
\end{equation}

Now we can take coordinates $(x, y)$ on $M$ with $\dx=\beta
T,~\dy=\alpha Q$ where $\beta,~\alpha$ are positive functions. From
(\ref{p1}) and the condition $[\dx, \dy]=0$, we find that
 \begin{align}
 &\beta_{y}=0,\label{p4}\\
 &\alpha_{x}=\alpha\beta\lambda\cot\theta.\nonumber
 \end{align}
Equation (\ref{p4}) implies that, after a change of the $x$-coordinate, we can assume $\beta=1$
and thus the metric takes the form
\begin{eqnarray*}
ds^{2}=dx^{2}+\alpha^{2}(x,y)dy^{2}.
\end{eqnarray*}

The Gauss and Ricci equation are, respectively, given by
\begin{align*}
({\widetilde R}(T, Q)T)^{\top}&=R(T, Q)T+A_{h(T, T)}Q-A_{h(Q, T)}T,\\
({\widetilde R}(T, Q)\eta)^{\perp}&=R^{\perp}(T, Q)\eta+h(A_{\eta}T, Q)-h(A_{\eta}Q, T),
\end{align*}
where
\begin{align*}
\widetilde R(X, Y)Z=&\big(\langle Y, Z\rangle-\langle Y,
\dt\rangle\langle Z, \dt\rangle\big)X-\big(\langle X,
Z\rangle-\langle X,
\dt\rangle\langle Z, \dt\rangle\big)Y\\
&-\big(\langle Y, Z\rangle\langle X, \dt\rangle-\langle X,
Z\rangle\langle Y, \dt\rangle\big)\dt,\forall X, Y, Z\in T(\S\times\R)\\
 R^{\perp}(T,
Q)\eta=&\big(\nabla^{\perp}_{T}\nabla^{\perp}_{Q}-\nabla^{\perp}_{Q}\nabla^{\perp}_{T}-\nabla^{\perp}_{[T,
Q]}\big)\eta.
\end{align*}

The Codazzi equations are
\begin{align*}
({\widetilde R}(T, Q)T)^{\perp}&=(\nabla^{\perp}_{T}h)(Q, T)-(\nabla^{\perp}_{Q}h)(T, T),\\
({\widetilde R}(T, Q)Q)^{\perp}&=(\nabla^{\perp}_{T}h)(Q,
Q)-(\nabla^{\perp}_{Q}h)(T, Q),
\end{align*}
where $(\nabla^{\perp}_{X}h)(Y, Z)=
\nabla^{\perp}_{X}\big(h(Y, Z)\big)-h(\nabla_{X}Y, Z)-h(Y, \nabla_{X}Z)$ for any $X, Y, Z\in TM.$

By a direct computation with (\ref{p1})--(\ref{n1}), the equations of Gauss, Ricci and Codazzi yield
  \begin{align}
 \lambda^{2}\cot^{2}\theta+\lambda_{x}\cot\theta+\cos^{2}\theta+\beta_{1}\beta_{3}-\beta_{2}^{2}&=0,
 \label{p6}\\
 \frac{ (\beta_{2})_{y}}{\alpha}+\lambda\cot\theta\sec^{2}\theta\beta_{1}-\lambda\cot\theta\beta_{3}-(\beta_{3})_{x}&=0,\label{p7}\\
 \frac{(\beta_{1})_{y}}{\alpha}-2\lambda\cot\theta\beta_{2}-(\beta_{2})_{x}&=0.\label{p8}
\end{align}

\section{Constant angle surfaces with parallel mean curvature vector}
In this section, we will discuss the constant angle surface $M$ with
parallel mean curvature vector in $\S\times \R$. In fact, we
have
\begin{thm}\label{thm1}
If $M$ is a constant angle surface in $\S\times \R$ with
parallel mean curvature vector $\vec{H}$, then $\vec{H}=0$, that is,
$M$ is a minimal surface in $\S\times \R$.
\end{thm}
\begin{proof}
Since the  mean curvature vector $\vec{H}$ of $M$ is parallel, that
is, $\nabla^{\perp}\vec{H}=0$, from (\ref{n1}), we have
\begin{align}
\lambda_{x}&=-(\beta_{1}+\beta_{3})\beta_{1}\tan\theta,\label{p9}\\
(\beta_{1})_{x}+(\beta_{3})_{x}&=\lambda\beta_{1}\tan\theta,\label{p10}
\end{align}
and
\begin{align}
\lambda_{y}&=-\alpha(\beta_{1}+\beta_{3})\beta_{2}\tan\theta,\label{p11}\\
(\beta_{1})_{y}+(\beta_{3})_{y}&=\alpha\lambda\beta_{2}\tan\theta.\label{p12}
\end{align}
From (\ref{p6}) and (\ref{p9}), we get
\begin{equation*}
\beta_{1}^{2}+\beta_{2}^{2}=\cot^{2}\theta(\lambda^{2}+\sin^{2}\theta)>0.
\end{equation*}
Thus we can set
\begin{equation}\label{p13}
\left\{\begin{aligned}
\beta_{1}&=\cot\theta\sqrt{\lambda^{2}+\sin^{2}\theta}\cos\gamma,\\
\beta_{2}&=\cot\theta\sqrt{\lambda^{2}+\sin^{2}\theta}\sin\gamma,
\end{aligned}\right.
\end{equation}
for some function $\gamma$ on $M$.

Taking the derivatives of (\ref{p13}), we
obtain
\begin{align}
(\beta_{1})_{x}&=-\beta_{2}\gamma_{x}+\frac{\lambda\lambda_{x}}
{\beta_{1}^{2}+\beta_{2}^{2}}\beta_{1}\cot^{2}\theta,\label{p15}\\
(\beta_{1})_{y}&=-\beta_{2}\gamma_{y}+\frac{\lambda\lambda_{y}}
{\beta_{1}^{2}+\beta_{2}^{2}}\beta_{1}\cot^{2}\theta,\label{p16}\\
(\beta_{2})_{x}&=\beta_{1}\gamma_{x}+\frac{\lambda\lambda_{x}}
{\beta_{1}^{2}+\beta_{2}^{2}}\beta_{2}\cot^{2}\theta,\label{p17}\\
(\beta_{2})_{y}&=\beta_{1}\gamma_{y}+\frac{\lambda\lambda_{y}}
{\beta_{1}^{2}+\beta_{2}^{2}}\beta_{2}\cot^{2}\theta.\label{p18}
\end{align}

Using (\ref{p9})--(\ref{p11}), (\ref{p15}) and (\ref{p18}), from (\ref{p7}) we get
\begin{eqnarray}
\frac{\beta_{1}}{\alpha}\gamma_{y}-\beta_{2}\gamma_{x}=2\lambda\beta_{3}\cot\theta.\label{p19}
\end{eqnarray}
Using (\ref{p9}), (\ref{p11}), (\ref{p16}) and (\ref{p17}), from (\ref{p8}) we get
\begin{eqnarray}
\frac{\beta_{2}}{\alpha}\gamma_{y}+\beta_{1}\gamma_{x}=-2\lambda\beta_{2}\cot\theta.\label{p20}
\end{eqnarray}
From (\ref{p19}) and (\ref{p20}) we have
\begin{equation}\label{p21}
\left\{\begin{aligned}
\gamma_{x}&=\frac{-2\lambda\cot\theta}
{\beta_{1}^{2}+\beta_{2}^{2}}\beta_{2}(\beta_{1}+\beta_{3}),\\
\gamma_{y}&=\frac{2\alpha\lambda\cot\theta}
{\beta_{1}^{2}+\beta_{2}^{2}}(\beta_{1}\beta_{3}-\beta_{2}^{2}).
\end{aligned}\right.
\end{equation}
Putting (\ref{p21}) into (\ref{p15})--(\ref{p18}), from (\ref{p9}), (\ref{p11}) and (\ref{p12}), we have
\begin{align*}
\lambda_{xy}&=-\tan\theta\Big[(\beta_{1})_{y}(\beta_{1}+\beta_{3})+\beta_{1}(\beta_{1}
+\beta_{3})_{y}\Big]\nonumber\\
&=-\tan\theta\left\{(\beta_{1}+\beta_{3})\Big[-\beta_{2}\gamma_{y}
-\frac{\alpha\lambda\cot\theta}{\beta_{1}^{2}+\beta_{2}^{2}}\beta_{1}\beta_{2}(\beta_{1}+\beta_{3})\Big]+\alpha\lambda\beta_{1}\beta_{2}\tan\theta\right\}\nonumber\\
&=\tan\theta\left\{(\beta_{1}+\beta_{3})\frac{\alpha\lambda\cot\theta}{\beta_{1}^{2}+\beta_{2}^{2}}
\Big[2\beta_{2}(\beta_{1}\beta_{3}-\beta_{2}^{2})+\beta_{1}\beta_{2}(\beta_{1}+\beta_{3})\Big]-\alpha\lambda\beta_{1}\beta_{2}\tan\theta\right\}
\nonumber\\
&=\beta_{2}(\beta_{1}+\beta_{3})\frac{\alpha\lambda}{\beta_{1}^{2}+\beta_{2}^{2}}(3\beta_{1}\beta_{3}-2\beta_{2}^{2}+\beta_{1}^{2})
-\alpha\lambda\beta_{1}\beta_{2}\tan^{2}\theta.
\end{align*}
Similarly, we also obtain
\begin{align*}
\lambda_{yx}&=-\tan\theta\Big[\alpha_{x}\beta_{2}(\beta_{1}+\beta_{3})
+\alpha\beta_{2}(\beta_{1}+\beta_{3})_{x}+\alpha(\beta_{2})_{x}(\beta_{1}+\beta_{3})\Big]\nonumber\\
&=-\tan\theta\Big[
\alpha\lambda\cot\theta\beta_{2}(\beta_{1}+\beta_{3})+\alpha\lambda\beta_{1}\beta_{2}\tan\theta-
\alpha(\beta_{1}+\beta_{3})\frac{\lambda\cot\theta}{\beta_{1}^{2}
+\beta_{2}^{2}}3\beta_{1}\beta_{2}(\beta_{1}+\beta_{3})
\Big]\nonumber\\
&=\beta_{2}(\beta_{1}+\beta_{3})\frac{\alpha\lambda}{\beta_{1}^{2}+\beta_{2}^{2}}(3\beta_{1}\beta_{3}+2\beta_{1}^{2}-\beta_{1}^{2})
-\alpha\lambda\beta_{1}\beta_{2}\tan^{2}\theta.
\end{align*}
Since $\alpha>0$, from the integrable condition
$\lambda_{xy}=\lambda_{yx}$, we have
\begin{eqnarray}
\lambda\beta_{2}(\beta_{1}+\beta_{3})=0.\label{con1}
\end{eqnarray}

We claim that $\lambda(p)=0$ for any $p\in M$.  Then from (\ref{p9}) and (\ref{p11}) we get $\beta_{1}+\beta_{3}=0$
since $\beta_{1}$ and $\beta_{2}$ cannot be zero simultaneously. Hence $M$ is minimal in $\S\times\R.$

To prove the claim, we discuss the equation (\ref{con1}) in two cases.

\textbf{Case 1}. $\beta_{2}\neq0$ at some point $p\in M$.

In this case, there exists a neighborhood $U$ of $p$ such that
$\lambda(\beta_{1}+\beta_{3})=0$ in $U$. If $\lambda(p)\neq0$, then there exists a neighborhood $ V\subset U$ such that
$\beta_{1}+\beta_{3}=0$ in $V$.  This contradicts (\ref{p12}).  Hence $\lambda(p)=0$.

\textbf{Case 2}. $\beta_{2}=0$ at some point $p\in M$.

First we assume that there exists a neighborhood $U$ of $p$ such that
$\beta_{2}=0$ in $U$. Then we get, in  $U$,
\begin{eqnarray*}
(\beta_{1})_{x}=-\lambda\cot\theta(\beta_{1}-\beta_{3})
\end{eqnarray*}
from (\ref{p7}) and (\ref{p10}).
On the other hand, from (\ref{p15}) and (\ref{p9})  we have, in  $U$,
\begin{eqnarray*}
(\beta_{1})_{x}=-\lambda\cot\theta(\beta_{1}+\beta_{3}).
\end{eqnarray*}
If $\lambda(p)\neq 0$, there exists a neighborhood $ V\subset U$  such that
$\lambda\neq 0$ in $V$. Then $\beta_{3}=0$ in $V$. Hence,
 $\beta_{1}=0$ in $V$ from (\ref{p7}). This contradicts $\beta_{1}^{2}+\beta_{2}^{2}>0$.  Hence $\lambda(p)=0$.

Otherwise,  there exists a sequence $\{q_{i}\}_{i=1}^{\infty}$ approaching $p$ such that $\beta_{2}(q_{i})\neq0$.
Then
$\lambda(q_{i})(\beta_{1}+\beta_{3})(q_{i})=0$. By taking the limit, $\lambda(p)(\beta_{1}+\beta_{3})(p)=0$.
If $\lambda(p)\neq0$, then $(\beta_{1}+\beta_{3})(p)=0$. From (\ref{con1}), there exists a neighborhood $U$ of $p$ such that
$\lambda\neq0$ in $U$, which implies
$\beta_{2}(\beta_{1}+\beta_{3})=0$ in $U$. Taking derivatives with respect to $x$ and $y$, using (\ref{p9})--(\ref{p12}), (\ref{p17}),
(\ref{p18}) and (\ref{p21}), we get
\begin{gather}
-\frac{\lambda\beta_{1}\beta_{2}(\beta_{1}+\beta_{3})^{2}\cot\theta}{\beta_{1}^{2}+\beta_{2}^{2}}+\lambda\beta_{1}\beta_{2}\tan\theta=0,\label{der1}\\
\frac{2\alpha\lambda\beta_{1}(\beta_{1}+\beta_{3})(\beta_{1}\beta_{3}-\beta_{2}^{2})\cot\theta}{\beta_{1}^{2}+\beta_{2}^{2}}+\alpha\lambda\beta_{2}^{2}\tan\theta=0.\label{der2}
\end{gather}
(\ref{der2})$\times\beta_{1}-$(\ref{der1})$\times\alpha\beta_{2}$, we have, in $ U$,
\begin{equation}\label{der3}
\frac{\alpha\lambda\cot\theta}{\beta_{1}^{2}+\beta_{2}^{2}}\beta_{1}(\beta_{1}+\beta_{3})
(2\beta_{1}^{2}\beta_{3}-\beta_{1}\beta_{2}^{2}+\beta_{3}\beta_{2}^{2})=0.
\end{equation}

Since $\beta_{2}(p)=0$, we can assume $\beta_{1}(p)>0$ without loss of generality. Hence $\beta_{3}(p)<0$ from $(\beta_{1}+\beta_{3})(p)=0$.  Then there exists a neighborhood $ V\subset U$  such that  $\beta_{1}>0, \beta_{3}<0$ in $V$. Thus in $V$, we have
$$2\beta_{1}^{2}\beta_{3}-\beta_{1}\beta_{2}^{2}+\beta_{3}\beta_{2}^{2}<0.$$
Then (\ref{der3}) implies that $\beta_{1}+\beta_{3}=0$ in $V$. This contradicts (\ref{p10}). Therefore, $\lambda(p)=0$.

Hence we have proved the claim and completed the proof of
Theorem \ref{thm1}.
\end{proof}

\section{Classification of minimal and constant angle surfaces}

In this section, we consider the minimal and constant angle surface $M$ in $\S\times\R$.
\begin{lem}\label{lem1}
Let $M$ be a minimal and constant angle surface in $\S\times\R$. Then the shape operators  with respect to
$\xi$ and $\eta$ are, respectively,
\begin{align*}
 A_{\xi}=\left(
 \begin{array}{cc}
 0 & 0\\
 0 & 0
 \end{array}\right),\quad
 A_{\eta}=\left(
 \begin{array}{cr}
 \beta_{1} & \beta_{2}\\
 \beta_{2} & -\beta_{1}
 \end{array}\right),
\end{align*}
where $\beta_1$ and $\beta_2$ are constants, satisfying
$\beta_1^2+\beta_2^2=\cos^2\theta$.
\end{lem}
\begin{proof}
From (\ref{CAS}) and the minimality of $M$ in $\S\times \R$,
the shape operator $ A_{\xi}$ associated to $\xi$ is
\begin{eqnarray}
 A_{\xi}=\left(
 \begin{array}{cc}
 0 & 0\\
 0 & 0
 \end{array}\right)
\end{eqnarray}
Hence, we have
\begin{eqnarray*}
\nabla_{T}T=\nabla_{T}Q= \nabla_{Q}T= \nabla_{Q}Q=0,
\end{eqnarray*}
which means that $M$ is flat. The coordinates $(x, y)$ on $M$ now can be chosen such that $\dx=T, \dy=Q$ (i.e.  $\alpha=1$).

From the minimality of $M$ in $\S\times \R$, the shape operator $A_\eta$ becomes
\begin{eqnarray*}
 A_{\eta}=\left(
 \begin{array}{cr}
 \beta_{1} & \beta_{2}\\
 \beta_{2} & -\beta_{1}
 \end{array}\right).
\end{eqnarray*}

The equations of Gauss, Ricci, and Codazzi (\ref{p6})--(\ref{p8}) are
  \begin{align*}
  &\beta_{1}^{2}+\beta_{2}^{2}=\cos^{2}\theta,\\
  &(\beta_{2})_{y}=-(\beta_{1})_{x},\\
  &(\beta_{1})_{y}=(\beta_{2})_{x}.
\end{align*}
The above equations yield that both $\beta_{1}$ and $\beta_{2}$ are
constant.
\end{proof}
Now let us consider
$\S\times\R$ as a hypersurface in $\E^5$ and denote $\dt$ by
$(0,0,0,0,1)$. We obtain the following classification theorem.
\begin{thm}\label{thm2}
A surface $M$ immersed in $\S\times\R$ is a minimal and constant angle
surface if and only if the immersion
\begin{eqnarray*}
   F \colon\quad M & \rightarrow &  \S\times\R \subset \E^5 \\
   (x,y) &\mapsto& F(x,y)
\end{eqnarray*}
is {\em(}up to isometries of \ $\S\times\R${\em)} locally given by
\begin{equation}\label{Sol}
\begin{aligned}
F(x,y)=&~(c_1\cos(\mu_1x+\nu_2y), c_1\sin(\mu_1x+\nu_2y), c_2\cos(\mu_2x-\nu_1y),\\
&~~~c_2\sin(\mu_2x-\nu_1y), x\sin\theta),
\end{aligned}
\end{equation}
where $\theta\in(0,\frac{\pi}{2})$ is the constant angle, $\nu_1\in
[1,1+\cos^2\theta]$ is a constant, and $\nu_2$, $\mu_1$, $\mu_2$,
$c_1$, $c_2$ are nonnegative constants given by
\begin{align*}
&\nu_2^2=\frac{1+\cos^2\theta-\nu_1^2}{1-\nu_1^2\sin^2\theta},~\mu_1^2=\frac{\nu_1^2\cos^4\theta}{1-\nu_1^2\sin^2\theta},~\mu_2^2=1+\cos^2\theta-\nu_1^2,\\
&c_1^2=\frac{1-\nu_1^2\sin^2\theta}{1+\cos^2\theta-\nu_1^2\sin^2\theta},~c_2^2=\frac{\cos^2\theta}{1+\cos^2\theta-\nu_1^2\sin^2\theta}.
\end{align*}
\end{thm}

\begin{proof}

First we prove that the given immersion (\ref{Sol}) is a minimal and
constant angle surface in $\S\times\R.$ To see this, we calculate
the tangent vectors
\begin{align*}
F_x=&(-\mu_1c_1\sin(\mu_1x+\nu_2y), \mu_1c_1\cos(\mu_1x+\nu_2y), -\mu_2c_2\sin(\mu_2x-\nu_1y),\\
&~~~~\mu_2c_2\cos(\mu_2x-\nu_1y), \sin\theta),\\
F_y=&(-\nu_2c_1\sin(\mu_1x+\nu_2y), \nu_2c_1\cos(\mu_1x+\nu_2y), \nu_1c_2\sin(\mu_2x-\nu_1y),\\
&~~-\nu_1c_2\cos(\mu_2x-\nu_1y), 0).
\end{align*}
The normal $N$ of $\S\times\R$ in $\E^5$ is
\begin{align*}
N=(c_1\cos(\mu_1x+\nu_2y), c_1\sin(\mu_1x+\nu_2y), c_2\cos(\mu_2x-\nu_1y), c_2\sin(\mu_2x-\nu_1y), 0).
\end{align*}
Let
\begin{align*}
\xi=&(\mu_1c_1\tan\theta\sin(\mu_1x+\nu_2y),-\mu_1c_1\tan\theta\cos(\mu_1x+\nu_2y), \mu_2c_2\tan\theta\sin(\mu_2x-\nu_1y),\\
&~~~-\mu_2c_2\tan\theta\cos(\mu_2x-\nu_1y),\cos\theta),\\
\eta=&(-c_2\cos(\mu_1x+\nu_2y),-c_2\sin(\mu_1x+\nu_2y),c_1\cos(\mu_2x-\nu_1y),c_1\sin(\mu_2x-\nu_1y),0).
\end{align*}
We can verify that $F_x$, $F_y$, $\xi$, $\eta$, $N$ are orthonormal
in $\E^{5}$. Thus $\{\xi, \eta\}$ is a basis of the normal plane of
$M$ in $\S\times\R$. Moreover, we have
$$\dt=\sin\theta F_x+\cos\theta\xi,$$
which means that the angle between $\dt$ and the normal plane is
constant $\theta$.

Furthermore, we can calculate the shape operators with respect to
$\xi$ and $\eta$ on $M$ in $\S\times\R$ respectively,
\begin{align*}
 A_{\xi}=\left(
 \begin{array}{cc}
 0 & 0\\
 0 & 0
 \end{array}\right),\quad
 A_{\eta}=\left(
 \begin{array}{cr}
 \beta_{1}' & \beta_{2}'\\
 \beta_{2}' & \beta_{3}'
 \end{array}\right),
\end{align*}
where
$$\beta_{1}'=-\beta_{3}'=\frac{(\nu_1^2-1)\cos\theta}{\sqrt{1-\nu_1^2\sin^2\theta}},\quad
\beta_{2}'=\frac{\nu_1\cos\theta\sqrt{1+\cos^2\theta-\nu_1^2}}{\sqrt{1-\nu_1^2\sin^2\theta}}.$$
Therefore, $M$ is a minimal surface in $\S\times\R$. Moreover, we
can see that $(\beta_{1}')^2+(\beta_{2}')^2=\cos^2\theta$.


Conversely, let us consider $M$ as an immersed surface
in $\E^5$ with codimension 3. Denote by $D,\dpp$ the Euclidean connection and the normal
connection of $M$ in $\E^5$, respectively. For the immersion
$F=(F_{1},F_{2},F_{3},F_{4},F_{5}):M\to\S\times\R\subset\E^{5}$, we
have three unit normals
\begin{align*}
N&=(F_{1},F_{2},F_{3},F_{4},0),\\
\xi&=(\xi_{1},\xi_{2},\xi_{3},\xi_{4},\cos\theta),\\
\eta&=(\eta_{1},\eta_{2},\eta_{3},\eta_{4},0),
\end{align*}
where $N$ is normal to $\S\times\R$ with the shape operator $\ta$.

For simplicity, we denote the first four components of a vector in $\E^5$ by adding
a tilde on it, say $F=(\tilde{F},F_5)$, etc.

Noticing  that $\langle T, \dt\rangle=(F_{5})_{x}=\sin\theta,\  \langle Q, \dt\rangle=(F_{5})_{y}=0,$ we can
take $F_{5}=x\sin\theta$ without loss of generality.

For any $X\in T_p M$, we have
\begin{align*}
\dpp_{X}N&=\langle D_{X}N, \xi\rangle\xi+\langle D_{X}N,
\eta\rangle\eta\\
&=\langle X-\langle X,\dt\rangle\dt,\xi\rangle\xi+\langle X-\langle
X,\dt\rangle\dt,\eta\rangle\eta\\
&=-\sin\theta\cos\theta\langle X, T\rangle\xi.
\end{align*}
By the Weingarten formula, we have
\begin{align}
\ta T&=-D_{T}N+\dpp_{T}N\nonumber\\
              &=-(\tilde{F}_{x},0)-\sin\theta\cos\theta(\tilde{\xi},\cos\theta),\label{eq04}\\
\ta Q&=-D_{Q}N+\dpp_{Q}N\nonumber\\
              &=-(\tilde{F}_{y},0).\nonumber
\end{align}
Thus the shape operator associated to $N$ is
\begin{eqnarray*}
 \ta=\left(
 \begin{array}{cc}
 -\sin^{2}\theta & 0\\
 0 & -1
 \end{array}\right).
\end{eqnarray*}
Comparing the first four components of (\ref{eq04}), we get
\begin{eqnarray*}
\xi_i=-\tan\theta(F_i)_x.
\end{eqnarray*}

Taking $(X, Y)=(T, T),  (T, Q), (Q, Q)$ in
$D_{X}Y=\tilde{\nabla}_{X}Y+\tilde{h}(X, Y)$, and $X=T,
Q$ in $D_{X}\eta=-\tilde{A}_{\eta}X+\tilde{\nabla}_{X}^{\perp}\eta$
respectively, we get the PDE system for $i=1,2,3,4$,
\begin{align}
\label{eq1} (F_i)_{xx}&=\beta_1\eta_i-\cos^2\theta F_i,\\
\label{eq2} (F_i)_{xy}&=\beta_2\eta_i,\\
\label{eq3} (F_i)_{yy}&=-\beta_1\eta_i-F_i,\\
\label{eq4} (\eta_i)_x&=-\frac{\beta_1}{\cos^2\theta}(F_i)_x-\beta_2(F_i)_y,\\
\label{eq5}
(\eta_i)_y&=-\frac{\beta_2}{\cos^2\theta}(F_i)_x+\beta_1(F_i)_y,
\end{align}
where $\beta_1$ and $\beta_2$ are as in Lemma \ref{lem1}. Obviously, the integrable conditions
are all satisfied. Moreover, we have $\xi_i=-\tan\theta(F_i)_x$ and
$F_5=x\sin\theta$, $\xi_5=\cos\theta$, $\eta_5=0$.

In the following, we will solve the above PDE system in three cases.

\textbf{Case 1.} $\beta_2=0$.

In this case, we can choose the direction of $\eta$ such that
$\beta_1=\cos\theta>0$, and then the PDE system becomes
\begin{align}
\label{a1} (F_i)_{xx}&=\cos\theta\eta_i-\cos^2\theta F_i,\\
\label{a2} (F_i)_{xy}&=0,\\
\label{a3} (F_i)_{yy}&=-\cos\theta\eta_i-F_i,\\
\label{a4} (\eta_i)_x&=-\frac{1}{\cos\theta}(F_i)_x,\\
\label{a5} (\eta_i)_y&=\cos\theta(F_i)_x.
\end{align}
From (\ref{a2}), we know that the solution has a separating form:
$F_i(x,y)=f_i(x)+g_i(y)$. Denote $\rho=\sqrt{1+\cos^2\theta}$. Taking
the derivative of (\ref{a1}) with respect to $x$ and using (\ref{a4}),
we get
\[f_i'''=-\rho^2f_i',\]
and then $f_i'(x)=k_i\cos(\rho x)+l_i\sin(\rho x)$. Taking the same operation
with respect to $y$, we find the solution has the form
\[F_i(x,y)=A_i\cos(\rho x)+B_i\sin(\rho x)+C_i\cos(\rho y)+D_i\sin(\rho y).\]
We can derive from (\ref{a1}) that
\[\eta_i(x,y)=-\frac{A_i}{\cos\theta}\cos(\rho x)-\frac{B_i}{\cos\theta}\sin(\rho x)+C_i\cos\theta\cos(\rho y)+D_i\cos\theta\sin(\rho y),\]
and we can also check that (\ref{a3})--(\ref{a5}) are all satisfied.

Since
\begin{align*}
(F_i)_x&=\rho\big(B_i\cos(\rho x)-A_i\sin(\rho x)\big),\\
(F_i)_y&=\rho\big(D_i\cos(\rho y)-C_i\sin(\rho y)\big),\\
\xi_i&=-\rho\tan\theta\big(B_i\cos(\rho x)-A_i\sin(\rho x)\big),
\end{align*}
and $F_x$, $F_y$ are orthonormal, we have
\begin{align*}
\cos^2\theta=\sum_i((F_i)_x)^2=&\rho^2\big(\sum_iB_i^2\cos^2(\rho x)+\sum_iA_i^2\sin^2(\rho x)-\sum_iA_iB_i\sin(2\rho x)\big),\\
1=\sum_i((F_i)_y)^2=&\rho^2\big(\sum_iD_i^2\cos^2(\rho y)+\sum_iC_i^2\sin^2(\rho y)-\sum_iC_iD_i\sin(2\rho y)\big),\\
0=\sum_i(F_i)_x(F_i)_y=&\rho^2\big(\sum_iB_iD_i\cos(\rho x)\cos(\rho y)+\sum_iA_iC_i\sin(\rho x)\sin(\rho y)\\
&-\sum_iB_iC_i\cos(\rho x)\sin(\rho y)-\sum_iA_iD_i\sin(\rho
x)\cos(\rho y)\big).
\end{align*}
Since $x$, $y$ are arbitrary, we have
\[\sum_iA_i^2=\sum_iB_i^2=\frac{\cos^2\theta}{\rho^2},\quad\sum_iC_i^2=\sum_iD_i^2=\frac{1}{\rho^2},\]
\[\sum_iA_iB_i=\sum_iC_iD_i=\sum_iB_iD_i=\sum_iA_iC_i=\sum_iB_iC_i=\sum_iA_iD_i=0,\]
and we can check that $F_x$, $F_y$, $\xi$, $\eta$ are orthonormal.
Hence, we have
\[\tilde{F}(x,y)=\frac{\cos\theta}{\rho}\cos(\rho x)\vec{e}_1+\frac{\cos\theta}{\rho}\sin(\rho x)\vec{e}_2+\frac{1}{\rho}\cos(\rho y)\vec{e}_3+\frac{1}{\rho}\sin(\rho y)\vec{e}_4.\]
where $\{\vec{e}_i\}_{i=1}^4$ is a fixed orthonormal basis of
$\mathbb{E}^4$. If we choose $\vec{e}_1=(1,0,0,0)$,
$\vec{e}_2=(0,1,0,0)$, $\vec{e}_3=(0,0,1,0)$,
$\vec{e}_4=(0,0,0,-1)$, the surface is locally given by
\begin{equation*}
F(x,y)=\left(\frac{\cos\theta}{\rho}\cos(\rho
x),\frac{\cos\theta}{\rho}\sin(\rho x),\frac{1}{\rho}\cos(\rho
y),-\frac{1}{\rho}\sin(\rho y),x\sin\theta\right).
\end{equation*}
This is the case $\nu_1=\rho=\sqrt{1+\cos^2\theta}$ (hence
$\mu_1=\rho$, $\mu_2=\nu_2=0$, $c_1=\frac{\cos\theta}{\rho}$,
$c_2=\frac{1}{\rho}$) in (\ref{Sol}).

\textbf{Case 2.} $\beta_1=0$.

In this case, we can choose the direction of $\eta$ such that
$\beta_2=\cos\theta>0$. The PDE system becomes
\begin{align}
\label{b1} (F_i)_{xx}&=-\cos^2\theta F_i,\\
\label{b2} (F_i)_{xy}&=\cos\theta\eta_i,\\
\label{b3} (F_i)_{yy}&=-F_i,\\
\label{b4} (\eta_i)_x&=-\cos\theta(F_i)_y,\\
\label{b5} (\eta_i)_y&=-\frac{1}{\cos\theta}(F_i)_x.
\end{align}
Solving (\ref{b1}) and (\ref{b3}), we find that the solution has the
form
\begin{align*}
F_i(x,y)=&~A_i\cos(x\cos\theta)\cos y+B_i\cos(x\cos\theta)\sin y\\
&+C_i\sin(x\cos\theta)\cos y+D_i\sin(x\cos\theta)\sin y.
\end{align*}
We can derive from (\ref{b2}) that
\begin{align*}
\eta_i=&~D_i\cos(x\cos\theta)\cos y-C_i\cos(x\cos\theta)\sin y\\
&-B_i\sin(x\cos\theta)\cos y+A_i\sin(x\cos\theta)\sin y,
\end{align*}
and we can check that (\ref{b4}) and (\ref{b5}) are satisfied.
Moreover, we have
\begin{align*}
(F_i)_x=~&\cos\theta(C_i\cos(x\cos\theta)\cos y+D_i\cos(x\cos\theta)\sin y\\
&-A_i\sin(x\cos\theta)\cos y-B_i\sin(x\cos\theta)\sin y),\\
(F_i)_y=&~B_i\cos(x\cos\theta)\cos y-A_i\cos(x\cos\theta)\sin y\\
&\,+D_i\sin(x\cos\theta)\cos y-C_i\sin(x\cos\theta)\sin y,\\
\xi_i=&~-\sin\theta\big(C_i\cos(x\cos\theta)\cos y+D_i\cos(x\cos\theta)\sin y\\
&\,-A_i\sin(x\cos\theta)\cos y-B_i\sin(x\cos\theta)\sin y\big).
\end{align*}
From the fact that $F_x$, $F_y$, $\xi$, $\eta$ are orthonormal,
a similar discussion as in Case 1 yields
\begin{align*}
\tilde{F}(x,y)=&\cos(x\cos\theta)\cos y\vec{e}_1+\cos(x\cos\theta)\sin y\vec{e}_2\\
&+\sin(x\cos\theta)\cos y\vec{e}_3+\sin(x\cos\theta)\sin y\vec{e}_4,
\end{align*}
where $\{\vec{e}_i\}_{i=1}^4$ is a fixed orthonormal basis of
$\mathbb{E}^4$. If we choose
$\vec{e}_1=(\frac{1}{\sqrt{2}},0,\frac{1}{\sqrt{2}},0)$,
$\vec{e}_2=(0,\frac{1}{\sqrt{2}},0,-\frac{1}{\sqrt{2}})$,
$\vec{e}_3=(0,\frac{1}{\sqrt{2}},0,\frac{1}{\sqrt{2}})$,
$\vec{e}_4=(-\frac{1}{\sqrt{2}},0,\frac{1}{\sqrt{2}},0)$, the
surface is locally given by
\begin{align*}
\begin{split}
F(x,y)=&\bigg(\frac{1}{\sqrt{2}}\cos(x\cos\theta+y),\frac{1}{\sqrt{2}}\sin(x\cos\theta+y),\frac{1}{\sqrt{2}}\cos(x\cos\theta-y),\\
&\,\frac{1}{\sqrt{2}}\sin(x\cos\theta-y),x\sin\theta\bigg).
\end{split}
\end{align*}
This is the case $\nu_1=1$ \big(hence $\mu_1=\mu_2=\cos\theta$,
$\nu_2=1$, $c_1=c_2=\frac{1}{\sqrt{2}}$\big) in (\ref{Sol}).

{\bf Case 3.} $\beta_1\beta_2\neq0$.

Taking the derivative of equation (\ref{eq1}) with respect to $x$,
and using equation (\ref{eq4}), we get
\[(F_i)_{xxx}=-\frac{\beta_1^2}{\cos^2\theta}(F_i)_x-\cos^2\theta(F_i)_x-\beta_1\beta_2(F_i)_y.\]
Taking the derivative with respect to $x$ again, and using equations
(\ref{eq2}), (\ref{eq1}), we get
\begin{equation}\label{eq6}
(F_i)_{xxxx}=\left(-\frac{\beta_1^2}{\cos^2\theta}-\beta_2^2-\cos^2\theta\right)(F_i)_{xx}-\beta_2^2\cos^2\theta
F_i.
\end{equation}
Similarly, taking the derivative of equation (\ref{eq3}) with respect to
$y$ twice, and using equations (\ref{eq5}), (\ref{eq2}),
(\ref{eq3}), we get
\[(F_i)_{yyy}=\frac{\beta_1\beta_2}{\cos^2\theta}(F_i)_x-\beta_1^2(F_i)_y-(F_i)_y,\]
and
\begin{equation}\label{eq7}
(F_i)_{yyyy}=\left(-\frac{\beta_2^2}{\cos^2\theta}-\beta_1^2-1\right)(F_i)_{yy}-\frac{\beta_2^2}{\cos^2\theta}F_i.
\end{equation}

The characteristic equation of (\ref{eq6}) is
\begin{equation}\label{eq6-1}
z^4+\left(\frac{\beta_1^2}{\cos^2\theta}+\beta_2^2+\cos^2\theta\right)z^2+\beta_2^2\cos^2\theta=0.
\end{equation}
Denote $b_1=\frac{\beta_1^2}{\cos^2\theta}+\beta_2^2+\cos^2\theta$
and $c_1=\beta_2^2\cos^2\theta$. Considering equation (\ref{eq6-1})
as a quadratic equation in $u=z^2$, the discriminant is
\[\Delta_1=b_1^2-4c_1=\frac{\beta_1^4}{\cos^4\theta}+\frac{2\beta_1^2\beta_2^2}{\cos^2\theta}+2\beta_1^2+\beta_1^4>0.\]
Since $c_1>0$, the two negative roots $u=-\mu_1^2$ and $u=-\mu_2^2$ of the equation are
\[-\mu_1^2=-\frac12(b_1+\sqrt{\Delta_1}),~~-\mu_2^2=-\frac12(b_1-\sqrt{\Delta_1}),\]
where we assume $\mu_1>0$, $\mu_2>0$.
\par Similarly, the characteristic equation of (\ref{eq7}) is
\begin{equation}\label{eq7-1}
w^4+\left(\frac{\beta_2^2}{\cos^2\theta}+\beta_1^2+1\right)w^2+\frac{\beta_2^2}{\cos^2\theta}=0.
\end{equation}
Denote $b_2=\frac{\beta_2^2}{\cos^2\theta}+\beta_1^2+1$ and
$c_2=\frac{\beta_2^2}{\cos^2\theta}$. Considering equation
(\ref{eq7-1}) as a quadratic equation as above, the discriminant is
\[\Delta_2=b_2^2-4c_2=\Delta_1>0 \]
and the two negative roots are
\[-\nu_1^2=-\frac12(b_2+\sqrt{\Delta_2}),~~-\nu_2^2=-\frac12(b_2-\sqrt{\Delta_2}),\]
where we assume $\nu_1>0$, $\nu_2>0$.

Now we denote $\Delta=\Delta_1=\Delta_2$. Since
$(F_i)_{xx}+(F_i)_{yy}=-(1+\cos^2\theta)F_i$ and
$\mu_1^2+\nu_2^2=\mu_2^2+\nu_1^2=1+\cos^2\theta$, the solution takes
the form
\begin{align*}
F_i(x,y)=&~c_1^{(i)}\cos(\mu_1x)\cos(\nu_2y)+c_2^{(i)}\cos(\mu_1x)\sin(\nu_2y)+c_3^{(i)}\sin(\mu_1x)\cos(\nu_2y)\\
&+c_4^{(i)}\sin(\mu_1x)\sin(\nu_2y)+c_5^{(i)}\cos(\mu_2x)\cos(\nu_1y)+c_6^{(i)}\cos(\mu_2x)\sin(\nu_1y)\\
&+c_7^{(i)}\sin(\mu_2x)\cos(\nu_1y)+c_8^{(i)}\sin(\mu_2x)\sin(\nu_1y).
\end{align*}
We can derive $\eta_i$ from (\ref{eq1}),
\begin{equation}\label{c1}
\begin{aligned}
\eta_i=&~\frac{1}{\beta_1}((F_i)_{xx}+\cos^2\theta F_i)\\
=&~\frac{\cos^2\theta-\mu_1^2}{\beta_1}\big(c_1^{(i)}\cos(\mu_1x)\cos(\nu_2y)+c_2^{(i)}\cos(\mu_1x)\sin(\nu_2y)\\
&+c_3^{(i)}\sin(\mu_1x)\cos(\nu_2y)+c_4^{(i)}\sin(\mu_1x)\sin(\nu_2y)\big)\\
&+\frac{\cos^2\theta-\mu_2^2}{\beta_1}\big(c_5^{(i)}\cos(\mu_2x)\cos(\nu_1y)+c_6^{(i)}\cos(\mu_2x)\sin(\nu_1y)\\
&+c_7^{(i)}\sin(\mu_2x)\cos(\nu_1y)+c_8^{(i)}\sin(\mu_2x)\sin(\nu_1y)\big).
\end{aligned}
\end{equation}
On the other hand, from (\ref{eq2})
\begin{equation}\label{c2}
\begin{aligned}
\eta_i=&~\frac{1}{\beta_2}(F_i)_{xy}\\
=&~\frac{\mu_1\nu_2}{\beta_2}\big(c_4^{(i)}\cos(\mu_1x)\cos(\nu_2y)-c_3^{(i)}\cos(\mu_1x)\sin(\nu_2y)\\
&-c_2^{(i)}\sin(\mu_1x)\cos(\nu_2y)+c_1^{(i)}\sin(\mu_1x)\sin(\nu_2y)\big)\\
&+\frac{\mu_2\nu_1}{\beta_2}\big(c_8^{(i)}\cos(\mu_2x)\cos(\nu_1y)-c_7^{(i)}\cos(\mu_2x)\sin(\nu_1y)\\
&-c_6^{(i)}\sin(\mu_2x)\cos(\nu_1y)+c_5^{(i)}\sin(\mu_2x)\sin(\nu_1y)\big).
\end{aligned}
\end{equation}

Comparing the first four terms, we find that
\begin{eqnarray*}
\frac{\cos^2\theta-\mu_1^2}{\beta_1}c_1^{(i)}=\frac{\mu_1\nu_2}{\beta_2}c_4^{(i)},\quad\frac{\cos^2\theta-\mu_1^2}{\beta_1}c_4^{(i)}=\frac{\mu_1\nu_2}{\beta_2}c_1^{(i)},\\
\frac{\cos^2\theta-\mu_1^2}{\beta_1}c_2^{(i)}=\frac{\mu_1\nu_2}{\beta_2}c_3^{(i)},\quad\frac{\cos^2\theta-\mu_1^2}{\beta_1}c_3^{(i)}=\frac{\mu_1\nu_2}{\beta_2}c_2^{(i)}.
\end{eqnarray*}

Since $\mu_1>0,\mu_2>0,\nu_1>0,\nu_2>0,$
\[2(\cos^2\theta-\mu_1^2)=\beta_1^2-\frac{\beta_1^2}{\cos^2\theta}-\sqrt{\Delta}<0,\]
\[2(\cos^2\theta-\mu_2^2)=\beta_1^2-\frac{\beta_1^2}{\cos^2\theta}+\sqrt{\Delta}>0,\]
we have that
\[(c_1^{(i)})^2=
(c_4^{(i)})^2,(c_2^{(i)})^2=(c_3^{(i)})^2.\]

Similarly, comparing the last four terms of (\ref{c1}) and (\ref{c2}), we obtain that
\[(c_5^{(i)})^2=(c_8^{(i)})^2,(c_6^{(i)})^2=(c_7^{(i)})^2.\]

Furthermore, we have for $\beta_1\beta_2>0$,
\[c_1^{(i)}=-c_4^{(i)},~c_2^{(i)}=c_3^{(i)},~c_5^{(i)}=c_8^{(i)},~c_6^{(i)}=-c_7^{(i)};\]
and for $\beta_1\beta_2<0$,
\[c_1^{(i)}=c_4^{(i)},~c_2^{(i)}=-c_3^{(i)},~c_5^{(i)}=-c_8^{(i)},~c_6^{(i)}=c_7^{(i)}.\]
\par Hence, for $\beta_1\beta_2>0$, we can set
\[F_i(x,y)=A_i\cos(\mu_1x+\nu_2y)+B_i\sin(\mu_1x+\nu_2y)+C_i\cos(\mu_2x-\nu_1y)+D_i\sin(\mu_2x-\nu_1y).\]
In fact, we can easily verify that the solution above satisfies the PDE system (\ref{eq1})--(\ref{eq5}).

Moreover, using the fact that $F_x$, $F_y$, $\xi$,
$\eta$ are orthonormal, we can derive that
\begin{align*}
\tilde{F}(x,y)=&~c_1\cos(\mu_1x+\nu_2y)\vec{e}_1+c_1\sin(\mu_1x+\nu_2y)\vec{e}_2\\
&+c_2\cos(\mu_2x-\nu_1y)\vec{e}_3+c_2\sin(\mu_2x-\nu_1y)\vec{e}_4
\end{align*}
where $\{\vec{e}_i\}_{i=1}^4$ is a fixed orthonormal basis of
$\mathbb{E}^4$, $c_1,c_2$ are positive constants  satisfying
$c_1^2=\frac{\nu_1^2-1}{\nu_1^2-\nu_2^2}$,
$c_2^2=\frac{1-\nu_2^2}{\nu_1^2-\nu_2^2}$. If we choose the natural
basis of $\mathbb{E}^4$, the surface is locally given by
\begin{align*}
F(x,y)=&~\big(c_1\cos(\mu_1x+\nu_2y),c_1\sin(\mu_1x+\nu_2y),c_2\cos(\mu_2x-\nu_1y),\\
&c_2\sin(\mu_2x-\nu_1y),x\sin\theta\big).
\end{align*}
This is the case $1<\nu_1<\sqrt{1+\cos^2\theta}$ in (\ref{Sol}).

Similarly, for $\beta_1\beta_2<0$, the surface is locally given by
\begin{align*}
F(x,y)=&~\big(c_1\cos(\mu_1x-\nu_2y),c_1\sin(\mu_1x-\nu_2y),c_2\cos(\mu_2x+\nu_1y),\\
&c_2\sin(\mu_2x+\nu_1y),x\sin\theta\big).
\end{align*}
If we change the coordinate to be $\{x,-y\}$, then this is the case
$1<\nu_1<\sqrt{1+\cos^2\theta}$ in (\ref{Sol}).

Here we need to derive the relations among the constants $\nu_1$,
$\nu_2$, $\mu_1$, $\mu_2$, $c_1$, $c_2$ when
$1<\nu_1<\sqrt{1+\cos^2\theta}$. In fact, by the definitions of
$\nu_1$ and $\nu_2$, we have
$\nu_1^2\nu_2^2=\frac{\beta_2^2}{\cos^2\theta}$ and
\begin{align*}
\nu_1^2+\nu_2^2&=\frac{\beta_2^2}{\cos^2\theta}+\beta_1^2+1\\
&=\nu_1^2\nu_2^2+\cos^2\theta-\cos^2\theta\nu_1^2\nu_2^2+1\\
&=\nu_1^2\nu_2^2\sin^2\theta+\cos^2\theta+1.
\end{align*}
Since $1+\cos^2\theta<\frac{1}{\sin^2\theta}$ when
$\theta\in(0,\frac{\pi}{2})$, we have
$\nu_2^2=\frac{1+\cos^2\theta-\nu_1^2}{1-\nu_1^2\sin^2\theta}$. By a
direct computation, we have
\begin{align*}
&\mu_1^2=1+\cos^2\theta-\nu_2^2=\frac{\nu_1^2\cos^4\theta}{1-\nu_1^2\sin^2\theta},\\
&\mu_2^2=1+\cos^2\theta-\nu_1^2,\\
&c_1^2=\frac{\nu_1^2-1}{\nu_1^2-\nu_2^2}=\frac{1-\nu_1^2\sin^2\theta}{1+\cos^2\theta-\nu_1^2\sin^2\theta},\\
&c_2^2=\frac{1-\nu_2^2}{\nu_1^2-\nu_2^2}=\frac{\cos^2\theta}{1+\cos^2\theta-\nu_1^2\sin^2\theta}.
\end{align*}
Hence we complete the proof of Theorem \ref{thm2}.
\end{proof}


\textbf{Acknowledgements}

The authors wish to  express their gratitude to Professors Haizhong
Li and Luc Vrancken for their suggestions and useful discussions.
This work was done partially while the first three authors visited
Departement Wiskunde, Katholieke Universiteit Leuven, Belgium. They
would like to thank the institute for its hospitality.


\end{document}